\documentclass[sn-mathphys-num]{sn-jnl}% Math and Physical Sciences Numbered Reference Style
\usepackage{graphicx}%
\usepackage{multirow}%
\usepackage{amsmath,amssymb,amsfonts}%
\usepackage{amsthm}%
\usepackage{mathrsfs}%
\usepackage[title]{appendix}%
\usepackage{xcolor}%
\usepackage{textcomp}%
\usepackage{manyfoot}%
\usepackage{algorithm}%
\usepackage{algorithmic}%
\usepackage{listings}%
\usepackage{epstopdf}
\usepackage{subfigure}
\usepackage{float}%Ìṩfloat¸¡¶¯»·¾³
\usepackage{booktabs}%ÌṩÃüÁî\toprule¡¢\midrule¡¢\bottomrule
\usepackage{multirow}%Ìṩ¿çÁÐÃüÁî\multicolumn{}{}{}
\usepackage{tabularx} % for 'tabularx' env. and 'X' col. type
\usepackage{ragged2e} % for \RaggedRight macro
\usepackage{booktabs} % for \toprule, \midrule etc macros
\usepackage{rotating}
\usepackage{enumitem}
\theoremstyle{thmstyleone}%
\newtheorem{theorem}{Theorem}[section]%  meant for continuous numbers
\newtheorem{example}{Example}[section]
\newtheorem{remark}{Remark}[section]
\newtheorem{lemma}{Lemma}[section]
\newtheorem{definition}{Definition}[section]

\numberwithin{equation}{section}

\begin{document}

\title[A reconsideration of quasimonotone variational inequality problems]{A reconsideration of quasimonotone variational inequality problems}


\author[1]{\fnm{Meiying} \sur{Wang}}\email{xdwangmeiying@163.com}

\author*[1]{\fnm{Hongwei} \sur{Liu}}\email{hwliuxidian@163.com}

\author[2]{\fnm{Jun} \sur{Yang}}\email{xysyyangjun@163.com}

\affil[1]{\orgdiv{School of Mathematics and Statistics}, \orgname{Xidian University}, \orgaddress{\city{Xi'an}, \postcode{710126}, \state{Shaanxi}, \country{China}}}

%\affil*[2]{\orgdiv{School of Mathematics and Statistics}, \orgname{Xidian University}, \orgaddress{\city{Xi'an}, \postcode{710126}, \state{Shaanxi}, \country{China}}}

\affil[2]{\orgdiv{School of Mathematics and Statistics}, \orgname{Xianyang Normal University}, \orgaddress{\city{Xianyang}, \postcode{712000}, \state{Shaanxi}, \country{China}}}


\abstract{This paper is based on Tseng's exgradient algorithm for solving variational inequality problems in real Hilbert spaces. Under the assumptions that the cost operator is quasimonotone and Lipschitz continuous, we establish the strong convergence, sublinear convergence, and Q-linear convergence of the algorithm. The results of this paper provide new insights into quasimonotone variational inequality problems, extending and enriching existing results in the literature. Finally, we conduct numerical experiments to illustrate the effectiveness and implementability of our proposed condition and algorithm.}
%Numerical experiments demonstrate the effectiveness and implementability of our proposed condition and algorithm.



\maketitle

\section{Introduction}\label{sec1}
Consider a real Hilbert space $H$ with inner product $\langle \cdot, \cdot \rangle$ and norm $\|\cdot\|$. The variational inequality problem associated with a continuous mapping $F : H \to H$, abbreviated as $VI(C, F)$, is defined as follows: find $u^\ell\in C$ such that
\begin{equation}\label{1.1}
  \langle F(u^\ell), z-u^\ell\rangle\geq 0,\quad \forall z\in C,
\end{equation}
where $C$ represents a nonempty closed convex subset of $H$. Let $S$ denote the solution set of $VI(C, F)$, and let $S_D$ denote the solution set of the following problem, which is specifically described as finding $u^\ell \in C$ such that
\begin{equation}\label{1.2}
  \langle F(z), z - u^\ell \rangle \geq 0, \quad \forall z \in C.
\end{equation}
Clearly, $S_D$ is a closed convex set, though it could be empty. In the case
where $F$ is a continuous mapping and $C$ is a convex set, it follows that $S_D \subseteq S$. Moreover, $S_D = S$ holds if $F$ is both pseudomonotone and continuous mapping \cite{co}. Nevertheless, it is essential to acknowledge that the reverse $S\subset S_D $ does not hold when $F$ is a quasimonotone and continuous mapping \cite{ye}.

To solve the variational inequality problem, various iterative algorithms have been proposed, such as \cite{ye,kor,sol,cen,tseng2000,mali,dong,yang2018,yang3,shehu,yao,Thong,Tan2023,thong2024,shehu2019}. One of these algorithms, known as the extragradient algorithm, was proposed by Korpelevich \cite{kor} (also by Antipin \cite{an}) for solving the variational inequality problem in finite dimensional spaces where the operator is both monotone and Lipschitz continuous. As we know, this method involves two projections onto the feasible set $C$. However, if the structure of $C$ is highly complex, the projection may be difficult to compute or have no explicit expression. \\
To overcome this limitation, one notable approach is the subgradient extragradient algorithm developed by Censor et al. \cite{cen}. This method substitutes the extragradient algorithm of using two projections onto $C$ with a new combination: one projection onto $C$ and another onto the half-space $T_n$. The details of this algorithm are as follows:
\begin{equation}\label{1.3}
\begin{cases}
w_n=P_C(z_n-\lambda F(z_n)),\\
z_{n+1}=P_{T_n}(z_n-\lambda F(w_n)),
\end{cases}
\end{equation}
where $\lambda\in (0, \frac{1}{L})$, $L$ is the Lipschitz constant of the monotone operator $F$ and $T_n=\{v\in H: \langle z_n-\lambda F(z_n)-w_n, v-w_n \rangle\leq0\}$.
Another method is the Tseng's extragradient algorithm proposed by Tseng in \cite{tseng2000}. This method replaces the second step in the extragradient algorithm with an explicit calculation step, as follows:
\begin{equation}\label{1.4}
\begin{cases}
w_n=P_C(z_n-\lambda F(z_n)),\\
z_{n+1}=w_n-\lambda (F(w_n)-F(z_n)),
\end{cases}
\end{equation}
where $\lambda\in (0, \frac{1}{L})$. \\%The author proved the strong convergence of the proposed method (\ref{1.4}).\\
The operators involved in the above algorithms (\ref{1.3}) and (\ref{1.4}) are monotone, and this requirement for the operators is too strict. Consequently, these methods were later used to solve $VI(C, F)$ with the cost operators being pseudomonotone and quasimonotone (see e.g., \cite{yao,Thong,Tan2023,yang3,shehu2019,thong2024}). In addition, the step sizes involved in these methods described above are related to the Lipschitz constant of the operator $F$, which is typically unknown or difficult to estimate. For this reason, Liu and Yang \cite{yang3} modified algorithm (\ref{1.4}) and proposed the following non-monotonic step size:
\begin{equation}\label{liuyang}
\lambda_{n+1}=
\begin{cases}
\min\{\frac{\mu\|z_n-w_n\|}{\|F(z_n)-F(w_n)\|},\ \lambda_n+\xi_n\},&\text{if}\ F(z_n)-F(w_n)\neq 0,\\
\lambda_n+\xi_n,&\text{otherwise},
\end{cases}
\end{equation}
where $\mu\in(0,1)$ and $\{\xi_n\}\subset [0,\infty)$ such that $\sum\limits_{n=1}^{\infty}\xi_n<\infty$.\\
It is worth noting that when dealing with quasimonotone variational inequality problems, the assumption that for all $p\in C$, $F(p)\neq 0$ is commonly made, as can be seen in \cite{uzor,ofem,mewomo2024,thong2024} and references therein. This assumption implies that the solution to $VI(C, F)$ is restricted to the set $S_D$. However, it is known that when the operator $F$ is quasimonotone, we have $S_D \subseteq S$, but the converse does not hold. That is to say, the solution of $VI(C, F)$ is not necessarily in $S_D$. Therefore, such an assumption is too restrictive. Then, when considering the general situation where the solution $u$ of the variational inequality problem satisfies $u\in S$, how to design algorithms and conditions becomes a difficult problem. As a result, the modified Tseng's extragradient algorithm was proposed by Liu and Yang \cite{yang3}, which incorporates the self-adaptive step size criterion (\ref{liuyang}) for solving the quasimonotone variational inequality problem. They proved that the proposed algorithm converges weakly to a point in $S$ under the following conditions:\\
(a)\ $S_D\neq \emptyset$.\\
(b)\ $F$ is quasimonotone and $L$-Lipschitz continuous mapping.\\
(c)\ $F$ is sequentially weakly continuous.\\
(d)\ $A=\{p\in C: F(p)=0\}\backslash S_D$ is a finite set.

The above method presents two noteworthy problems as follows:\\
(P1)\ The requirement of condition (d) is too strict and does not apply to the case where the set $A$ contains infinitely many points.\\
(P2)\ Under the given assumptions, the authors only obtained weak convergence results. However, it is widely acknowledged that strong convergence results are generally more preferable than weak convergence.

Our goal in this paper is to solve the two aforementioned problems. Specifically, we provide a more relaxed condition compared to \cite{yang3}, which also allows the solution of the quasimonotone variational inequality problem to belong to the solution set $S$, rather than restricting it to the solution set $S_D$. Additionally, under a certain condition, we achieve strong convergence results for the algorithm. And we delve into the analysis of the convergence rate, including sublinear convergence and $Q$-linear convergence.

The organization of the paper is as follows: As for the definitions and lemmas to be applied in the subsequent proofs, we present them in Section \ref{sec2}. In Section \ref{sec3}, we introduce the algorithm and assumptions for solving quasimonotone variational inequality problems. In Section \ref{sec4}, we conduct a comprehensive analysis of the strong convergence of the algorithm under the given assumptions. Then, sublinear convergence and $Q$-linear convergence are obtained in Section \ref{sec5}. Finally, Section \ref{sec55} demonstrates the effectiveness of the algorithm and the feasibility of the proposed condition through numerical experiments.

\section{Preliminaries}\label{sec2}
In this section, we provide an introduction to some concepts and lemmas that will prove to be instrumental in the ensuing analysis. We use $u_n \rightharpoonup u$ to denote that the sequence $\{u_n\}$ converges weakly to $u$, and $u_n \to u$ to denote that $\{u_n\}$ converges strongly to $u$.
\begin{definition}
A mapping $F: H\to H$ is referred to as
\\ \emph{(i)} $\delta$-strongly monotone if there exists a positive constant $\delta$ such that
\begin{equation*}
  \langle F(u)-F(z), u-z \rangle\geq\delta\|u-z\|^2,\quad \forall u,\ z\in H.
\end{equation*}
\\ \emph{(ii)} monotone if
\begin{equation*}
  \langle F(u)-F(z), u-z \rangle\geq0,\quad \forall u,\ z\in H.
\end{equation*}
\\ \emph{(iii)} $\eta$-strongly pseudomonotone if there exists a positive constant $\eta$ such that
\begin{equation*}
  \langle F(u), z-u \rangle\geq 0\Longrightarrow \langle F(z), z-u \rangle\geq\eta\|u-z\|^2,\quad \forall u,\ z\in H.
\end{equation*}
\\ \emph{(iv)} pseudomonotone if
\begin{equation*}
  \langle F(u), z-u \rangle\geq0\Longrightarrow \langle F(z), z-u \rangle\geq0,\quad \forall u,\ z\in H.
\end{equation*}
\\ \emph{(v)} quasimonotone if
\begin{equation*}
  \langle F(u), z-u \rangle>0\Longrightarrow \langle F(z), z-u \rangle\geq0,\quad \forall u,\ z\in H.
\end{equation*}
\\ \emph{(vi)} $L$-Lipschitz continuous, if there exists a constant $L>0$ such that
\begin{equation*}
  \|F(u)-F(z)\|\leq L\|u-z\|,\quad \forall u,\ z\in H.
\end{equation*}
\end{definition}
Obviously, (i)$\Rightarrow$(ii)$\Rightarrow$(iv)$\Rightarrow$(v) and (iii)$\Rightarrow$(iv)$\Rightarrow$(v). However, the converse do not necessarily hold true.

For a given point $w$, we define its metric projection onto the set $C$ as $P_C(w)$. This projection is distinguished by the characteristic that it minimizes the distance between $w$ and any point in $C$, and is formally defined as:
\begin{equation*}
  P_C(w):=\mathop{\arg\min}\{\|w-v\|:v\in C\}.
\end{equation*}
%The following lemmas will be applied in the convergence analysis.
\begin{lemma}\label{lm:2.1}\cite{bau}
The following results hold for any $w,\ v\in H$ and any $u\in C$:
\\ \emph{(i)} $\langle w-P_C(w),\ u-P_C(w)\rangle\leq0$.
\\ \emph{(ii)} $\|w-P_C(w)\|^2\leq \langle w-P_C(w),\ w-u \rangle$.
\\ \emph{(iii)} $\|P_C(w)-P_C(v)\|\leq\|w-v\|$.
\end{lemma}
%\begin{lemma}\label{lm:2.2}
%The following equations hold for any $w,\ v,\ z\in H$ and $m\in \mathbb{R}$:\\
%\emph{(i)} $||w\pm v||^2=||w||^2\pm2\langle w,\ v \rangle+||v||^2$;\\
%\emph{(ii)} $||(1-m)w+mv||^2=(1-m)||w||^2+m||v||^2-m(1-m)||w-v||^2$.
%\end{lemma}

\begin{lemma}\label{lm:2.3}\cite{opial1967weak}(Opial)
If $v_n\rightharpoonup v$ for any sequence $\{v_n\}$ in $H$, then\\
\begin{equation*}
  \liminf\limits_{n\to\infty}\|v_n-v\|<\liminf\limits_{n\to\infty}\|v_n-p\|,\quad\forall v\neq p.
\end{equation*}
\end{lemma}

\section{Proposed Method}\label{sec3}
In this section, we first present the assumptions throughout the paper as follows:
\begin{description}
 \item[(A1)]$S_D\neq \emptyset$.
 \item[(A2)]$F: H\to H$ is quasimonotone and $L$-Lipschitz continuous mapping with $L>0$.
 \item[(A3)]The mapping $F$ satisfies whenever $\{u_n\}\subset H$ and $u_n\rightharpoonup u$, there is $\|Fu\|\leq \liminf\limits_{n\to\infty}\|Fu_n\|$.
 \item[(A4)]The set $A:=\{p\in C: F(p)=0\}\backslash S_D$. For all $u\in S_D$, define $A_u=\{u\}\cup A$ that satisfies:
   $\exists\ \delta_u>0$, $\forall\ y\in A_u$, $\exists\  r_i\in H$ and $\|r_i\|=1$, $i=1,2, \cdots, m$ ($m\in \mathbb{N}$), let
   \begin{equation*}
     \Omega(y, \delta_u)=\{x\in H:|\langle r_i, x-y \rangle|<\delta_u,\ i=1,\cdots,m\},
   \end{equation*}
   such that for any $y_1$, $y_2\in A_u$ and $y_1\neq y_2$, $\Omega(y_1, \delta_u)\cap \Omega(y_2, \delta_u)=\emptyset$.
\end{description}
\begin{remark}
If the condition in \cite{yang3} holds, that is, assuming that $A=\{p\in C: F(p)=0\}\backslash S_D$ is a finite set, then our proposed condition (A4) naturally holds. Here, let $A=\{y^1, y^2,\ldots, y^t\}$ and $u=y^0$, then $A_u = \{y^0, y^1, y^2,\ldots, y^t\}$. We define $r_{ij}=\frac{y^j-y^i}{\|y^j-y^i\|}$ and $\delta_{ij}=\frac{1}{4}|\langle r_{ij}, y^j-y^i \rangle|=\frac{1}{4}\|y^j-y^i\|$, where $i,j\in\{0,1,2,\ldots,t\}$ and $j\neq i$. Then take
\begin{equation*}
  \delta_u= \min_{\substack{j \neq i}} \delta_{ij}=\frac{1}{4}\min_{\substack{j \neq i}}\|y^j-y^i\|>0.
\end{equation*}
Next, we verify that $\Omega(y^j, \delta_u)\cap \Omega(y^i, \delta_u)=\emptyset$, $\forall y^j,\ y^i\in A_u$ and $y^j\neq y^i$. Assume that the conclusion does not hold, i.e., there exists $x\in \Omega(y^j, \delta_u)\cap \Omega(y^i, \delta_u) $. Thus, we can get
\begin{equation*}
  |\langle r_{ij}, x-y^j \rangle|<\delta_u\ \text{and}\ |\langle r_{ij}, x-y^i \rangle|<\delta_u.
\end{equation*}
Therefore,
\begin{align*}
  4\delta_u\leq4\delta_{ij}&=\|y^j-y^i\|\\&=|\langle r_{ij}, y^j-y^i \rangle|\\
  &\leq|\langle r_{ij}, x-y^i \rangle|+|\langle r_{ij}, x-y^j \rangle|<2\delta_u.
\end{align*}
That's a contradiction. Hence, $\Omega(y^j, \delta_u)\cap \Omega(y^i, \delta_u)=\emptyset$. This implies that assumption (A4) holds.\\
Conversely, when (A4) holds, the condition in \cite{yang3} does not necessarily hold. This is due to the fact that our proposed condition (A4) can be adapted to the case where the set $A$ contains an infinite number of points.
\end{remark}

To demonstrate the applicability of condition (A4) in the case where the set $A$ contains infinitely many points, we provide the following example.
\begin{example}
Let $C=[0,+\infty)$ and
\begin{equation*}
  F(z)=1+\sin(z).
\end{equation*}
It is easy to know that $F$ is a quasimonotone and Lipschitz continuous mapping on $C$. Furthermore, it follows that
\begin{equation*}
  S=\{0\}\cup\{2k\pi+\frac{3\pi}{2},k=0,1,2,\cdots\}\ \text{and}\ S_D=\{0\}.
\end{equation*}
Therefore, $A=\{p\in C: F(p)=0\}\backslash S_D=\{2k\pi+\frac{3\pi}{2},k=0,1,2,\cdots\}$ is an infinite set and $A_u=\{u\}\cup A=\{0\}\cup \{2k\pi+\frac{3\pi}{2},k=0,1,2,\cdots\}$. Take $\delta_u=\frac{\pi}{2}$ and $r_i=1$, $i=1,2, \cdots, m$. In this case, for all $y\in A_u$,
\begin{equation*}
  \Omega(y, \delta_u)=\{x\in H:|x-y|<\frac{\pi}{2}\}.
\end{equation*}
Then, for all $y_1,\ y_2\in A_u$ and $y_1\neq y_2$, $\Omega(y_1, \delta_u)\cap \Omega(y_2, \delta_u)=\emptyset$.
\end{example}

The following is the self-adaptive extragradient algorithm proposed by Liu and Yang \cite{yang3} for solving $VI(C,F)$ (\ref{1.1}).
\renewcommand{\thealgorithm}{3.1}
\begin{algorithm}[H]
\caption{Self-adaptive Tseng's Extragradient Algorithm}
\label{AL1}
\begin{algorithmic}
\STATE{\emph{Step 0:} Let $\mu\in(0,1)$ and $\{\xi_k\}\subset [0,\infty)$ such that $\sum\limits_{k=1}^{\infty}\xi_k<\infty$. Choose arbitrarily initial point $u_1\in H$ and $\lambda_1>0$.} %Let $\{\phi_k\}\subset [1,\infty)$ and $\{\xi_k\}\subset [0,\infty)$ such that $\sum\limits_{k=1}^{\infty}(\phi_k-1)<\infty$ and $\sum\limits_{k=1}^{\infty}\xi_k<\infty$.
\STATE{\emph{Step 1:} Compute
\begin{equation*}
z_n=P_C(u_n-\lambda_nF(u_n)).
\end{equation*}
If $u_n=z_n$ or $F(z_n)=0$, then terminate the process ($z_n$ is a solution). Otherwise,
}
\STATE{\emph{Step 2:} Compute
\begin{equation*}
  u_{n+1}=z_n+\lambda_n(F(u_n)-F(z_n)).
\end{equation*}
\STATE{\emph{Step 3:} Update
\begin{equation*}
\lambda_{n+1}=
\begin{cases}
\min\{\frac{\mu\|u_n-z_n\|}{\|F(u_n)-F(z_n)\|},\ \lambda_n+\xi_n\},\ \ \ \text{if}\ F(u_n)-F(z_n)\neq 0,\\
\lambda_n+\xi_n,\qquad\qquad\qquad\qquad\qquad\ \text{otherwise}.
\end{cases}
\end{equation*}
}
%$\eta_n=\sigma_n\frac{||(I-P_Q)Ax_n||^2+||(I-P_K)Bx_n||^2}
%{||A^*(I-P_Q)Ax_n+B^*(I-P_K)Bx_n||^2},\ 0<\sigma_n<4.$
}
\STATE{\emph{Step 4:}\ Set $n:= n + 1$ to update $n$, and then return to
\emph{Step 1}.}
\end{algorithmic}
\end{algorithm}

\section{Strong Convergence}\label{sec4}
In this section, we analyze the strong convergence of the sequence $\{u_n\}$ generated by Algorithm \ref{AL1}, assuming that assumptions \emph{(A1)}-\emph{(A4)} and the following condition are satisfied:
\begin{itemize}
  \item\emph{(A5)} The sequence $\{z_n\}\subset C$ and $z_n\rightharpoonup z$, for $\varepsilon\geq0$, there is $\liminf\limits_{n\to\infty}\frac{|\langle F(z_n), z_n-z \rangle|}{\|z_n-z\|^{2+\varepsilon}}>0$.
\end{itemize}

\begin{lemma}\label{lll}
The sequence $\{\lambda_n\}$ is generated by Algorithm \ref{AL1}. Then it follows that $\{\lambda_n\}$ has upper bound $\lambda_1+\Xi$ and lower bound $\min\{\lambda_1, \frac{\mu}{L}\}$. And $\lim\limits_{n\to\infty}\lambda_n=\lambda$ with $\lambda\in[\min\{\lambda_1, \frac{\mu}{L}\}, \lambda_1+\Xi]$, where $\Xi=\Sigma_{n=1}^\infty \xi_n$.
\end{lemma}

\begin{proof}
A detailed proof can be obtained from Lemma 3.1 in \cite{yang3}.
\end{proof}

\begin{lemma}\label{22}
Assume that conditions (A1) and (A2) hold. Then, the sequence $\{u_n\}$ generated by Algorithm \ref{AL1} is bounded, $\lim\limits_{n\to\infty}\|u_n-u\|$ exists for any $u\in S_D$, $\lim\limits_{n\to\infty}\|u_n-z_n\|=0$ and $\lim\limits_{n\to\infty}\|u_{n+1}-u_n\|=0$.
\end{lemma}
\begin{proof}
Refer to Lemma 3.2 in \cite{yang3} for a detailed proof.
\end{proof}

\begin{lemma}\label{333}
Assume that conditions (A1)-(A3) hold. Let $\{u_n\}$ be a sequence generated by Algorithm \ref{AL1}. If there exists a subsequence $\{u_{n_k}\}$ of $\{u_n\}$ that converges weakly to $z\in H$, then $z\in S_D$ or $F(z) = 0$.
\end{lemma}
\begin{proof}
For the proof, see Lemma 3.3 in \cite{yang3}.
\end{proof}

\begin{lemma}\label{444}
Assume that conditions (A1)-(A4) hold. The sequence $\{u_n\}$ generated by Algorithm \ref{AL1} has no more than one weak cluster point in $S_D$. And there exists $N_1 \geq 0$ such that for all $n \geq N_1$, $u_n\in \Omega$, where $\Omega = \bigcup_{y\in A_u}\Omega(y, \frac{\delta_u}{2})$.
\end{lemma}
\begin{proof}
First, we show that there is no more than one weak cluster point in $S_D$ for the sequence $\{u_n\}$.\\
Suppose the sequence $\{u_n\}$ has no less than two weak cluster points in $S_D$, denoted as $\tilde{u}$ and $\check{u}$, where $\tilde{u}\neq\check{u}$. Let $\{u_{n_j}\}$ be a subsequence of $\{u_n\}$ that converges weakly to $\tilde{u}$ as $j \to \infty$. According to Lemma \ref{22}, $\lim\limits_{n\to\infty}\|u_n-u\|$ exists for any \( u \in S_D \). Furthermore, based on Lemma \ref{lm:2.3}, we obtain
\begin{align*}
  \lim\limits_{n\to\infty}\|u_n-\tilde{u}\|
  =&\lim\limits_{j\to\infty}\|u_{n_j}-\tilde{u}\|=\liminf\limits_{j\to\infty}\|u_{n_j}-\tilde{u}\|\\
  <&\liminf\limits_{j\to\infty}\|u_{n_j}-\check{u}\|=\lim\limits_{n\to\infty}\|u_n-\check{u}\|\\
  =&\lim\limits_{l\to\infty}\|u_{n_l}-\check{u}\|=\liminf\limits_{l\to\infty}\|u_{n_l}-\check{u}\|\\
  <&\liminf\limits_{l\to\infty}\|u_{n_l}-\tilde{u}\|=\lim\limits_{l\to\infty}\|u_{n_l}-\tilde{u}\|
  =\lim\limits_{n\to\infty}\|u_{n}-\tilde{u}\|.
\end{align*}
This leads to a contradiction. Hence, the conclusion holds. Moreover, according to Lemma \ref{333} and the definition of $A_u$, it is clear that $A_u$ contains all the weak cluster points of the sequence $\{u_n\}$.\\
Next, we prove that $u_n\in\Omega$. Suppose there exists a subsequence $\{u_{n_k}\}$ of $\{u_n\}$ such that $u_{n_k}\notin \Omega$ for any $k$. Given that the sequence $\{u_{n_k}\}$ is bounded, there exists a subsequence of $\{u_{n_k}\}$ which converges weakly to $z^*$. Without loss of generality, we still denote this subsequence as $\{u_{n_k}\}$, which implies that $u_{n_k}\rightharpoonup z^*$. Since $\Omega$ is an open set and by assumption $u_{n_k} \notin \Omega$, it follows that $z^* \notin \Omega$, that is $z^*\notin A_u$. This is impossible. Therefore, we conclude that $u_n\in \Omega$ when $n$ is sufficiently large, i.e., there exists $N_1 \geq 0$ such that for all $n \geq N_1$, $u_n\in \Omega$.
\end{proof}

\begin{lemma}\label{555}
Assume that conditions (A1)-(A4) hold. Then, the sequence $\{u_n\}$ that is generated by Algorithm \ref{AL1} converges weakly to a point of $S$.
\end{lemma}

\begin{proof}
From Lemma \ref{22} it follows that $\lim\limits_{n\to\infty}\|u_{n+1}-u_n\|=0$, then $\exists N_2 > N_1$, $\forall n\geq N_2$, such that
\begin{equation}\label{4.01}
  \|u_{n+1}-u_n\|<\frac{\delta_u}{4}.
\end{equation}
Assume that $\{u_n\}$ has weak cluster points with more than one. Let $y_1$ and $y_2$ denote any two cluster points of $\{u_n\}$ and $y_1\neq y_2$. It follows from Lemma \ref{444} that $\exists N_3 \geq N_2>N_1$, we have $u_{N_3}\in \Omega(y_1, \frac{\delta_u}{2})$ and $u_{N_3+1}\in \Omega(y_2, \frac{\delta_u}{2})$.
That is for $i=1,\cdots,m$,
\begin{equation}\label{4.1}
  -\frac{\delta_u}{2}<\langle r_i, u_{N_3}-y_1 \rangle<\frac{\delta_u}{2}.
\end{equation}
Since $u_{N_3+1}\in \Omega(y_2, \frac{\delta_u}{2})$, i.e., $u_{N_3+1}\in \Omega(y_2, \delta_u)$, and according to $\Omega(y_1, \delta_u)\cap \Omega(y_2, \delta_u)=\emptyset$, it follows that $u_{N_3+1}\notin \Omega(y_1, \delta_u)$. It implies that $\exists\ r_j$, $j\in \{1,\cdots m\}$ such that
\begin{equation}\label{4.2}
  \langle r_j, u_{N_3+1}-y_1 \rangle\geq\delta_u\ \text{or}\ \langle r_j, u_{N_3+1}-y_1\rangle\leq -\delta_u.
\end{equation}
For $i = j$ in (\ref{4.1}), combined with (\ref{4.2}), we get
\begin{align*}
 \frac{\delta_u}{2}&\leq|\langle r_j, u_{N_3+1}-u_{N_3} \rangle|\\
  &\leq\|r_j\|\|u_{N_3+1}-u_{N_3}\|\\
  &=\|u_{N_3+1}-u_{N_3}\|.
\end{align*}
Then, from (\ref{4.01}), we can obtain
\begin{equation*}
  \|u_{N_3+1}-u_{N_3}\|< \frac{\delta_u}{4}.
\end{equation*}
Consequently, we deduce
\begin{equation*}
  \frac{\delta_u}{2}\leq\|u_{N_3+1}-u_{N_3}\|< \frac{\delta_u}{4},
\end{equation*}
which clearly leads to a contradiction. Therefore, there is only one weak cluster point of the sequence $\{u_n\}$ in $S$, i.e., $u_n\rightharpoonup u^\ell$.
\end{proof}

\begin{theorem}\label{th3.1}
Assume that conditions (A1)-(A5) are satisfied. The sequence $\{u_n\}$ generated by Algorithm \ref{AL1} converges strongly to a point $u^\ell\in S$.
\end{theorem}
\begin{proof}
From $u_{n+1}=z_n-\lambda_n(F(z_n)-F(u_n))$, we have
\begin{align}\label{3.1}
  \|u_{n+1}-u^\ell\|^2=&\|z_n-\lambda_n(F(z_n)-F(u_n))-u^\ell\|^2\notag\\
  =&\|z_n-u^\ell\|^2+\lambda_n^2\|F(z_n)-F(u_n)\|^2-2\lambda_n\langle F(z_n)-F(u_n), z_n-u^\ell \rangle\notag\\
  =&\|u_n-u^\ell\|^2+\|z_n-u_n\|^2+2\langle z_n-u_n, u_n-u^\ell\rangle+\lambda_n^2\|F(z_n)-F(u_n)\|^2\notag\\
  &-2\lambda_n\langle F(z_n)-F(u_n), z_n-u^\ell\rangle\notag\\
  =&\|u_n-u^\ell\|^2+\|z_n-u_n\|^2-2\langle z_n-u_n, z_n-u_n\rangle+2\langle z_n-u_n, z_n-u^\ell\rangle\notag\\
  &+\lambda_n^2\|F(z_n)-F(u_n)\|^2-2\lambda_n\langle F(z_n)-F(u_n), z_n-u^\ell\rangle.
\end{align}
According to $z_n = P_C(u_n-\lambda_nF(u_n))$ and $u^\ell\in S\subseteq C$, we obtain
\begin{equation*}
  \langle z_n-u_n+\lambda_nF(u_n),  z_n-u^\ell\rangle\leq 0,
\end{equation*}
which implies that
\begin{equation}\label{3.2}
  \langle z_n-u_n,  z_n-u^\ell\rangle\leq -\lambda_n\langle F(u_n), z_n-u^\ell\rangle.
\end{equation}
Thus, by (\ref{3.1}), (\ref{3.2}) and the definition of $\{\lambda_{n+1}\}$, we have
\begin{align}\label{3.31}
  \|u_{n+1}-u^\ell\|^2\leq&\|u_{n}-u^\ell\|^2-\|z_n-u_n\|^2+\lambda_n^2\|F(z_n)-F(u_n)\|^2-2\lambda_n\langle F(z_n), z_n-u^\ell\rangle\notag\\
  \leq&\|u_{n}-u^\ell\|^2-(1-\frac{\mu^2\lambda_n^2}{\lambda_{n+1}^2})\|z_n-u_n\|^2-2\lambda_n\langle F(z_n), z_n-u^\ell \rangle.
\end{align}
Since $\lambda_n\to \lambda$ as $n\to\infty$ and $\mu\in(0,1)$, it follows that $\lim\limits_{n\to\infty}(1-\frac{\mu^2\lambda_n^2}{\lambda_{n+1}^2})=1-\mu^2>0$.
Therefore, there exists $N>N_3$ such that for any $n\geq N$, there is
\begin{equation}\label{3.32}
  1-\frac{\mu^2\lambda_n^2}{\lambda_{n+1}^2}
  >\frac{1}{2}(1-\mu^2)>0.
\end{equation}
From the fact that $u_n\rightharpoonup u^\ell$ and $\lim\limits_{n\to\infty}\|u_n-z_n\|=0$, it follows that $z_n\rightharpoonup u^\ell$. Hence, by assumption \emph{(A5)}, we have $\exists N'>N$ and $c>0$, $\forall n\geq N'$ such that
\begin{equation}\label{3.3}
  |\langle F(z_n), z_n-u^\ell\rangle|\geq c\|z_n-u^\ell\|^{2+\varepsilon}\geq 0.
\end{equation}

For the purpose of discussion, we will divide the proof into two cases as follows for discussion.\\
\textbf{\emph{Case I}}\ If $u^\ell\in S_D$, then by the fact that $z_n\in C $, we have
\begin{equation*}
  \langle F(z_n), z_n-u^\ell \rangle\geq 0.
\end{equation*}
Hence, combining with (\ref{3.3}), we get
\begin{equation}\label{3.4}
  \langle F(z_n), z_n-u^\ell\rangle\geq c\|z_n-u^\ell\|^{2+\varepsilon}\geq 0.
\end{equation}
As a result, in accordance with (\ref{3.31}), (\ref{3.32}) and (\ref{3.3}), we obtain that for any $n\geq N'$,
\begin{align*}
  \|u_{n+1}-u^\ell\|^2\leq&\|u_{n}-u^\ell\|^2-2\lambda_n\langle F(z_n), z_n-u^\ell \rangle\notag\\
  \leq&\|u_{n}-u^\ell\|^2-2\tilde{\lambda}c\|z_n-u^\ell\|^{2+\varepsilon},
\end{align*}
%Furthermore, it means that $\|x_{n+1}-x^\ell\|\leq \|x_{n}-x^\ell\|$. Therefore, $\lim\limits_{n\to\infty}\|x_n-x^\ell\|$ exists.
where $\tilde{\lambda}:=\min\{\frac{\mu}{L}, \lambda_1\}$. And we obtain
\begin{equation*}
  2\tilde{\lambda}c\|z_n-u^\ell\|^{2+\varepsilon}\leq\|u_{n}-u^\ell\|^2-\|u_{n+1}-u^\ell\|^2.
\end{equation*}
Therefore,
\begin{equation*}
  2\tilde{\lambda}c\sum\limits_{k=N}^{n}\|z_k-u^\ell\|^{2+\varepsilon}\leq\|u_{N}-u^\ell\|^2-\|u_{n+1}-u^\ell\|^2.
\end{equation*}
As $\{u_n\}$ is bounded, we can get $\sum\limits_{k=N}^{\infty}\|z_k-u^\ell\|^{2+\varepsilon}<\infty$. Thus, $\lim\limits_{n\to\infty}\|z_n-u^\ell\|=0$.
Hence, combining this with the fact that $\lim\limits_{n\to\infty}\|u_n-z_n\|=0$ of Lemma \ref{22}, we get
\begin{equation*}
  \|u_n-u^\ell\|\leq\|u_n-z_n\|+\|z_n-u^\ell\|\to0,\ n\to\infty.
\end{equation*}
%Then, from Lemma \ref{22} we have $\lim\limits_{n\to\infty}\|x_n-u\|$ exists, and from
\textbf{\emph{Case II}}\ If $u^\ell\notin S_D$, that is, $u^\ell\in A$. Next, we prove $\|F(z_n)\|\to 0$ as $n\to\infty$.\\ Since $u_{n}\rightharpoonup u^\ell$ and $\|u_{n}-z_{n}\|\to 0$ as $n\to\infty$, then $z_{n}\rightharpoonup u^\ell\in C$. Based on the Lipschitz continuity of $F$ and the sequence $\{z_{n}\}$ is bounded, it follows that $0\leq\limsup\limits_{n\to\infty}\|F(z_n)\|<+\infty$.\\
Assume that $\limsup\limits_{n\to\infty}\|F(z_n)\|\neq0$, i.e., $\limsup\limits_{n\to\infty}\|F(z_n)\|>0$. For the sake of generality, it may be useful to take $\lim\limits_{n\to\infty}\|F(z_n)\|=K>0$. Therefore, for all $n\geq N'$, $\|F(z_n)\| > \frac{K}{2}$ holds. In accordance with the definition of $\{z_n\}$ and Lemma \ref{lm:2.1}(i), we get
\begin{equation*}
  \langle u_n-\lambda_nF(u_n)-z_n, p-z_n \rangle\leq0,\ \forall p\in C,
\end{equation*}
which is equivalent to
\begin{equation*}
  \frac{1}{\lambda_n}\langle u_n-z_n, p-z_n \rangle\leq\langle F(u_n), p-z_n\rangle,\ \forall p\in C.
\end{equation*}
As a consequence, we can derive
\begin{equation}\label{1}
  \frac{1}{\lambda_n}\langle u_n-z_n, p-z_n\rangle-\langle F(u_n)-F(z_n), p-z_n\rangle\leq\langle F(z_n), p-z_n\rangle,\ \forall p\in C.
\end{equation}
Since $\lim\limits_{n\to\infty}\|u_n-z_n\|=0$ and $F$ is $L$-Lipschitz continuous mapping, it follows that
\begin{equation*}
  \lim\limits_{n\to\infty}\|F(u_n)-F(z_n)\|\leq\lim\limits_{n\to\infty}L\|u_n-z_n\|=0.
\end{equation*}
From the fact that $\{z_n\}$ is bounded and $\lim\limits_{n\to\infty}\lambda_n=\lambda>0$, and by making $n\to\infty$ in (\ref{1}) we get
\begin{equation}\label{2}
  0\leq\liminf\limits_{n\to\infty}\langle F(z_n), p-z_n\rangle\leq\limsup\limits_{n\to\infty}\langle F(z_n), p-z_n\rangle<+\infty.
\end{equation}
If $\limsup\limits_{n\to\infty}\langle F(z_n), p-z_n\rangle>0$, then there exists a subsequence $\{z_{n_k}\}$ of $\{z_n\}$ such that $\lim\limits_{k\to\infty}\langle F(z_{n_k}), p-z_{n_k}\rangle>0$. Thus there exists $k_0$ such that for all $k\geq k_0$, $\langle F(z_{n_k}), p-z_{n_k}\rangle>0$. Since $F$ is quasimonotone, we have $\langle F(p), p-z_{n_k}\rangle\geq 0$ for all $k \geq k_0$. Therefore, as $k \to \infty$, it follows that $\langle F(p), p - u^\ell \rangle \geq 0$, and this implies $u^\ell \in S_D$. \\
If $\limsup\limits_{n\to\infty}\langle F(z_n), p-z_n\rangle=0$, then from (\ref{2}), we have
\begin{equation*}
  \lim\limits_{n\to\infty}\langle F(z_n), p-z_n\rangle=\liminf\limits_{n\to\infty}\langle F(z_n), p-z_n\rangle=\liminf\limits_{n\to\infty}\langle F(z_n), p-z_n\rangle=0.
\end{equation*}
We select the sequence $\{\rho_n\}$ that is positive and decreasing, with $\lim_{n \to \infty} \rho_n = 0$. Define $\theta_n = |\langle F(z_n), p-z_n\rangle| + \rho_n > 0$ and $\omega_n = \frac{F(z_n)}{\|F(z_n)\|^2}$, for all $n \geq N'$, which ensures that $\langle F(z_n), \omega_n\rangle = 1$. Thus, we have $\langle F(z_n), p+\theta_n\omega_n - z_n \rangle > 0$, $\forall n \geq N'$. Given that $F$ is quasimonotone, it follows that $\langle F(p + \theta_n \omega_n), p+\theta_n\omega_n - z_n \rangle \geq 0$. This implies that for all $n \geq N'$,
\begin{align}\label{3}
  \langle F(p), p+\theta_n\omega_n-z_n\rangle=&\langle F(p)-F(p+\theta_n\omega_n), p+\theta_n\omega_n-z_n\rangle\notag\\
  &+\langle F(p+\theta_n\omega_n), p+\theta_n\omega_n-z_n\rangle\notag\\
  \geq&\langle F(p)-F(p+\theta_n\omega_n), p+\theta_n\omega_n-z_n\rangle\notag\\
  \geq&-\|F(p)-F(p+\theta_n\omega_n)\|\|p+\theta_n\omega_n-z_n\|\notag\\
  \geq&-\theta_nL\|\omega_n\|\|p+\theta_n\omega_n-z_n\|\notag\\
  =&-\theta_n\frac{L}{\|F(z_n)\|}\|p+\theta_n\omega_n-z_n\|\notag\\
  \geq&-\theta_n\frac{2L}{K}\|p+\theta_n\omega_n-z_n\|.
\end{align}
Since $\lim\limits_{n\to\infty}\theta_n=0$ and $\{\|p + \theta_n\omega_n - z_n\|\}$ is bounded, then for all $p \in C$, $\langle F(p), p - u^\ell \rangle \geq 0$ holds by taking the limit in (\ref{3}). Consequently, we have $u^\ell \in S_D$.

In conclusion, combining the two different situations described above, we get $u^\ell \in S_D$, which contradicts the assumption that $u^\ell \notin S_D$. Therefore, we obtain $\limsup\limits_{n\to\infty}\|F(z_n)\|=0$. Thus,
\begin{equation*}
 \lim\limits_{n\to\infty}\|F(z_n)\|=\limsup\limits_{n\to\infty}\|F(z_n)\|=\liminf\limits_{n\to\infty}\|F(z_n)\|=0.
\end{equation*}
Then, according to (\ref{3.3}), we obtain
\begin{align*}
  c\|z_n-u^\ell\|^{2+\varepsilon}\leq|\langle F(z_n), z_n-u^\ell\rangle|\leq\|F(z_n)\|\|z_n-u^\ell\|.
\end{align*}
We suppose that the algorithm does not stop at a finite step, that is, $\|z_n-u^\ell\|\neq0$. Thus, we have
\begin{equation*}
  c\|z_n-u^\ell\|^{1+\varepsilon}\leq\|F(z_n)\|.
\end{equation*}
Hence, by $\lim\limits_{n\to\infty}\|F(z_n)\|=0$, we get $\|z_n-u^\ell\|\to0$ as $n\to\infty$. Thus, based on the fact that $\lim\limits_{n\to\infty}\|u_n-z_n\|=0$, we obtain
\begin{equation*}
  \|u_n-u^\ell\|\leq\|u_n-z_n\|+\|z_n-u^\ell\|\to0,\ n\to\infty.
\end{equation*}
Finally, we obtain that the sequence $\{u_n\}$ converges strongly to $u^\ell\in S$.
\end{proof}

\begin{remark}
It is easy to observe that if $F$ is $\eta$-strongly pseudomontone mapping, then our proposed condition (A5) is satisfied. However, the converse does not hold, meaning that condition (A5) is weaker than $\eta$-strong pseudomontone. For example, take $H=\mathbb{R}$, $C = [-1, 1]$ and $F$ to be the following quasimonotone and continuous mapping as considered in \cite{yang3}:
\begin{equation*}
F(z)=
\begin{cases}
2z-1,&z>1,\\
z^2,&z\in [-1,1],\\
-2z-1,&z<-1.
\end{cases}
\end{equation*}
Consequently, we get $S_D=\{-1\}$ and $S=\{0,-1\}$. We can take $u^\ell\in S$(i.e., $z_n\rightarrow u^\ell$) and $\varepsilon=1$ in condition (A5). Since the algorithm is assumed not to stop at a finite step, $\|z_n-u^\ell\|\neq0$. Define
\begin{align*}
  I_1&=\{n\in \mathbb{N}: z_n>1\},\\
  I_2&=\{n\in \mathbb{N}: z_n\in [-1,1]\},\\
  I_3&=\{n\in \mathbb{N}: z_n<-1\}.
\end{align*}
If $z_n\to 0$ as $n\to\infty$, that is, $u^\ell=0\notin S_D$, then, by above \emph{Case II} in Theorem \ref{th3.1}, we have $\lim\limits_{n\to\infty}\|F(z_n)\|=0$. This indicates that there exists $N$ such that for all $n\geq N$, we have $n\in I_2$. Therefore, we can easily obtain
\begin{equation*}
  \frac{|\langle F(z_n), z_n-u^\ell \rangle|}{\|z_n-u^\ell\|^{3}}=\frac{|F(z_n)|}{\|z_n-u^\ell\|^{2}}=\frac{|z_n|^2}{\|z_n\|^{2}}=1>0,
\end{equation*}
i.e.,
\begin{equation*}
  \liminf\limits_{n\to\infty}\frac{|\langle F(z_n), z_n-u^\ell \rangle|}{\|z_n-u^\ell\|^{3}}>0.
\end{equation*}
If $z_n\to -1$ as $n\to\infty$, that is, $u^\ell=-1$, then we have $\lim\limits_{n\to\infty}\|F(z_n)\|=1$. Hence,
\begin{equation*}
  \frac{|\langle F(z_n), z_n-u^\ell \rangle|}{\|z_n-u^\ell\|^{3}}=\frac{|F(z_n)|}{\|z_n-u^\ell\|^{2}}>0,
\end{equation*}
and
\begin{equation*}
  \liminf\limits_{n\to\infty}\frac{|\langle F(z_n), z_n-u^\ell \rangle|}{\|z_n-u^\ell\|^{3}}=\liminf\limits_{n\to\infty}\frac{|F(z_n)|}{\|z_n-u^\ell\|^{2}}=+\infty.
\end{equation*}
This demonstrates that condition (A5) is satisfied for a quasimonotone mapping $F$. And it is easy to verify that $F$ is not a $\eta$-strongly pseudomontone mapping. Indeed, take $u=-1$ and $v=0$, we have $\langle F(u), v-u \rangle=1$ and $\langle F(v), v-u \rangle=0$.
%This distinction is further demonstrated through numerical experiments presented in \textcolor[rgb]{0.00,0.50,1.00}{Section 4}.
\end{remark}

\section{Convergence Rate}\label{sec5}
In this section, we aim to analyze the convergence rate of the sequence $\{u_n\}$ which is generated by Algorithm \ref{AL1}, provided that conditions \emph{(A1)}-\emph{(A4)} are satisfied. Additionally, we set $\varepsilon = 0$ in condition \emph{(A5)} and take the following assumption into reconsideration:
\begin{itemize}
  \item \emph{(A5')}\ The sequence $\{z_n\}\subset C$ and $z_n\rightharpoonup z$, for $0\leq\alpha\leq1$, there is $\liminf\limits_{n\to\infty}\frac{n^\alpha\langle F(z_n), z_n-z \rangle}{\|z_n-z\|^2}>0$.
\end{itemize}
\subsection{Sublinear Convergence}
We analyze the sublinear convergence rate of the sequence $\{u_n\}$ in this subsection. In assumption \emph{(A5')}, if $0< \alpha\leq1$, then we can derive the following theorem.
\begin{theorem}
Assume that conditions (A1)-(A4) and (A5') are satisfied. Let $\{u_n\}$ be the sequence generated by Algorithm \ref{AL1}. The following two results hold:\\
\emph{(i)}  If $0<\alpha<1$, then $\{u_n\}$ converges sublinearly to a point $u^\ell\in S$ at any order.\\
\emph{(ii)}  If $\alpha=1$, then $\{u_n\}$ converges sublinearly to a point $u^\ell\in S$.
\end{theorem}
\begin{proof}
As $u_n \rightharpoonup u^\ell$ and $\|u_n - z_n\| \to 0$ as $n\to\infty$, we deduce that $z_n\rightharpoonup u^\ell$. Noting that $\{z_n\}\subset C$, then by assumption \emph{(A5')}, we have $\exists \tilde{c}>0$, $\forall n\geq N'>N$,
\begin{equation}\label{321}
  \langle F(z_n), z_n-u^\ell\rangle\geq \frac{\tilde{c}}{n^\alpha}\|z_n-u^\ell\|^{2}.
\end{equation}
According to (\ref{3.32}), we get for all $n\geq N'> N$, $1-\frac{\mu^2\lambda_n^2}{\lambda_{n+1}^2}
>\frac{1-\mu^2}{2}>0$. We use
$\tilde{\lambda}>0$ to denote the lower bound of
$\{\lambda_n\}$. And since $\lim\limits_{n\to\infty}\frac{1}{n^\alpha}=0$, there exists $\bar{N}>N'$ such that $0<\frac{1}{n^\alpha}<\frac{1-\mu^2}{4\tilde{\lambda}\tilde{c}}$, $\forall n\geq \bar{N}$. Moreover,
\begin{align*}
  \|u_n-u^\ell\|^2=&\|u_n-z_n+z_n-u^\ell\|^2 \\
  =&\|u_n-z_n\|^2+\|z_n-u^\ell\|^2+2\langle u_n-z_n, z_n-u^\ell \rangle\\
  \leq&\|u_n-z_n\|^2+\|z_n-u^\ell\|^2+2\|u_n-z_n\|\|z_n-u^\ell\|\\
  \leq&2\|u_n-z_n\|^2+2\|z_n-u^\ell\|^2.
\end{align*}
Therefore, substituting (\ref{321}) into (\ref{3.31}) yields for all $n\geq \bar{N}$,
\begin{align}\label{325}
  \|u_{n+1}-u^\ell\|^2\leq&\|u_{n}-u^\ell\|^2-(1-\frac{\mu^2\lambda_n^2}{\lambda_{n+1}^2})\|z_n-u_n\|^2
  -\frac{2\lambda_n\tilde{c}}{n^\alpha}\|z_n-u^\ell\|^{2}\notag\\
  \leq&\|u_{n}-u^\ell\|^2-\frac{1-\mu^2}{2}\|z_n-u_n\|^2
  -\frac{2\tilde{\lambda}\tilde{c}}{n^\alpha}\|z_n-u^\ell\|^{2}\notag\\
  \leq&\|u_{n}-u^\ell\|^2-\min\{\frac{1-\mu^2}{2},\frac{2\tilde{\lambda}\tilde{c}}{n^\alpha}\}
  (\|z_n-u_n\|^2+\|z_n-u^\ell\|^{2})
  \notag\\
  \leq&\|u_{n}-u^\ell\|^2-\frac{1}{2}\min\{\frac{1-\mu^2}{2},2\tilde{\lambda}\tilde{c}\frac{1}{n^\alpha}\}
  \|u_{n}-u^\ell\|^2\notag\\
  =&\|u_{n}-u^\ell\|^2-\frac{\tilde{\lambda}\tilde{c}}{n^\alpha}
  \|u_{n}-u^\ell\|^2\notag\\
  =&(1-\frac{\tau}{n^\alpha})\|u_n-u^\ell\|^2,
\end{align}
where $\tau:=\tilde{\lambda}\tilde{c}$. \\
%As a result, there exists $\bar{N}>N'$ such that for all $n\geq \bar{N}$,
%\begin{align}
  %\|u_{n+1}-u^\ell\|^2\leq&\|u_{n}-u^\ell\|^2-\frac{1}{2}\min\{1-\frac{\mu^2\lambda_n^2}{\lambda_{n+1}^2},2\tilde{\lambda}\tilde{c}\frac{1}{n^\alpha}\}
  %\|u_{n}-u^\ell\|^2\notag\\
  %=&(1-\frac{\tau}{n^\alpha})\|u_n-u^\ell\|^2,
%\end{align}
Let $a_n:=\|u_n-u^\ell\|^2$. Therefore, from (\ref{325}) we have
\begin{align*}
  a_{n+1}\leq&(1-\frac{\tau}{n^\alpha})a_n\leq\prod\limits_{k=\bar{N}}^{n}(1-\frac{\tau}{k^\alpha})a_{\bar{N}}.
\end{align*}
Thus,
\begin{equation*}\label{326}
  \ln a_{n+1}\leq\sum\limits_{k=\bar{N}}^{n}\ln(1-\frac{\tau}{k^\alpha})+\ln a_{\bar{N}}.
\end{equation*}
From the fact that $\ln(1-t)<-t$ when $t\in[0,1)$, we have
\begin{equation}\label{327}
  \ln a_{n+1}\leq\sum\limits_{k=\bar{N}}^{n}(-\frac{\tau}{k^\alpha})+\ln a_{\bar{N}}.
\end{equation}
(i) If $0<\alpha<1$, we have
\begin{equation}\label{328}
  \int_{k}^{k+1}\frac{1}{k^\alpha}dt\geq\int_{k}^{k+1}\frac{1}{s^\alpha}ds=\frac{1}{1-\alpha}s^{1-\alpha}\Big|_{k}^{k+1}
  =\frac{1}{1-\alpha}[(k+1)^{1-\alpha}-k^{1-\alpha}].
\end{equation}
Therefore, by (\ref{327}) and (\ref{328}), we can get
\begin{equation*}
   \ln a_{n+1}\leq-\frac{\tau}{1-\alpha}[(n+1)^{1-\alpha}-\bar{N}^{1-\alpha}]+\ln a_{\bar{N}}.
\end{equation*}
Thus, we have
\begin{equation*}
  a_{n+1}=\|u_{n+1}-u^\ell\|^2\leq e^{-\frac{\tau}{1-\alpha}[(n+1)^{1-\alpha}-\bar{N}^{1-\alpha}]}\cdot a_{\bar{N}},
\end{equation*}
i.e.,
\begin{equation*}
 \|u_{n+1}-u^\ell\|\leq e^{-\frac{\tau}{2(1-\alpha)}[(n+1)^{1-\alpha}-\bar{N}^{1-\alpha}]}\cdot (a_{\bar{N}})^\frac{1}{2}=o(\frac{1}{n^r}).
\end{equation*}
That is $\lim\limits_{n\to\infty}\frac{e^{-\frac{\tau}{1-\alpha}[(n+1)^{1-\alpha}-\bar{N}^{1-\alpha}]}}{\frac{1}{n^r}}=0$ ($\forall r>0$), we obtain that $\{u_n\}$ converges sublinearly to $u^\ell$ at any order.\\
(ii) If $\alpha=1$, we get
\begin{equation*}
  \int_{k}^{k+1}\frac{1}{k}dt\geq\int_{k}^{k+1}\frac{1}{s}ds=\ln s\Big|_{k}^{k+1}
  =\ln(k+1)-\ln k.
\end{equation*}
Then, from (\ref{327}), we obtain
\begin{equation*}
  \ln a_{n+1}\leq\sum\limits_{k=\bar{N}}^{n}(-\frac{\tau}{k})+\ln a_{\bar{N}}\leq-\tau(\ln(n+1)-\ln \bar{N})+\ln a_{\bar{N}}
  =-\ln(n+1)^\tau+\ln \bar{N}^\tau+\ln a_{\bar{N}}.
\end{equation*}
%\begin{align*}
%  \ln a_{n+1}\leq&\sum\limits_{k=\bar{N}}^{n}(-\frac{\tau}{k})+\ln a_1\\
%  \leq&-\tau(\ln(n+1)-\ln \bar{N})+\ln a_1\\
%   =&-\ln(n+1)^\tau+\ln \bar{N}^\tau+\ln a_1.
%\end{align*}
Hence, we have
\begin{equation*}
  a_{n+1}=\|u_{n+1}-u^\ell\|^2\leq e^{-\ln(n+1)^\tau}\cdot\bar{N}^\tau\cdot a_{\bar{N}}=\frac{1}{(n+1)^\tau}\bar{N}^\tau a_{\bar{N}}.
\end{equation*}
That is, for $r=\frac{\tau}{2}$, we obtain
\begin{equation*}
  \|u_{n+1}-u^\ell\|\leq (\frac{1}{n+1})^{\frac{\tau}{2}}\bar{N}^{\frac{\tau}{2}} (a_{\bar{N}})^{\frac{1}{2}}=O(\frac{1}{n^r}).
\end{equation*}
Thus the sublinear convergence in this case is related to $\tau$, i.e., to $\tilde{\lambda}$ and $\tilde{c}$.
\end{proof}

\subsection{Linear Convergence}
We analyze the linear convergence rate of the sequence $\{u_n\}$ in this subsection. In assumption \emph{(A5')}, if $\alpha=0$, then we can derive the following theorem.
\begin{theorem}
Assume that conditions (A1)-(A4) and (A5') are satisfied. The sequence $\{u_n\}$ generated by Algorithm \ref{AL1} $Q$-linear convergence to a point $u^\ell\in S$.
\end{theorem}
\begin{proof}
Since $u_n\rightharpoonup u^\ell$ and $\|u_n-z_n\|\to0$ as $n\to\infty$, we obtain $z_n\rightharpoonup u^\ell$.
And by the fact that $\{z_n\}\subset C$, it is clear that assumption \emph{(A5')} holds. Thus, when $\alpha=0$ in assumption \emph{(A5')}, we have $\forall n\geq N'>N$,
\begin{equation}\label{329}
  \langle F(z_n), z_n-u^\ell\rangle\geq \tilde{c}\|z_n-u^\ell\|^{2}.
\end{equation}
According to (\ref{3.32}), we get for all $n\geq N'> N$, $1-\frac{\mu^2\lambda_n^2}{\lambda_{n+1}^2}
>\frac{1-\mu^2}{2}>0$. Therefore, it follows from (\ref{3.31}) and (\ref{329}) that for all $n\geq N'$,
\begin{align*}
  \|u_{n+1}-u^\ell\|^2\leq&\|u_{n}-u^\ell\|^2-(1-\frac{\mu^2\lambda_n^2}{\lambda_{n+1}^2})\|z_n-u_n\|^2
  -2\lambda_n\tilde{c}\|z_n-u^\ell\|^{2}\notag\\
  \leq&\|u_{n}-u^\ell\|^2-\frac{1-\mu^2}{2}\|z_n-u_n\|^2
  -2\lambda_n\tilde{c}\|z_n-u^\ell\|^{2}\notag\\
  \leq&\|u_{n}-u^\ell\|^2-\frac{1}{2}\min\{\frac{1-\mu^2}{2},2\tilde{\lambda}\tilde{c}\}
  \|u_{n}-u^\ell\|^2\notag\\
  \leq&(1-\frac{1}{2}\min\{\frac{1-\mu^2}{2},2\tilde{\lambda}\tilde{c}\})
  \|u_{n}-u^\ell\|^2,
\end{align*}
where $\tilde{\lambda}$ is the lower bound of $\{\lambda_n\}$. Thus, $1-\beta:=1-\frac{1}{2}\min\{\frac{1-\mu^2}{2},2\tilde{\lambda}\tilde{c}\}\in(0,1)$.\\
Thus, for all $n\geq N'$,
\begin{equation*}
  \|u_{n+1}-u^\ell\|^2\leq(1-\beta)\|u_n-u^\ell\|^2,
\end{equation*}\
i.e.,
\begin{equation*}
  \|u_{n+1}-u^\ell\|\leq\sqrt{(1-\beta)}\|u_n-u^\ell\|.
\end{equation*}\
%And by induction, it is concluded that
%\begin{equation*}
%  \|x_{n+1}-x^\ell\|^2\leq(1-\beta)^{n-N'}\|x_N'-x^\ell\|^2\leq(1-\beta)^{n}b,
%\end{equation*}
%where $b=\frac{\|x_N'-x^\ell\|^2}{(1-\beta)^{N'}}$.
Therefore, the conclusion holds.

\end{proof}

\section{Numerical Experiments}\label{sec55}
In this section, the effectiveness and superiority of the algorithm under the conditions presented in this paper are demonstrated by presenting numerical experiments. ``Iter'' is the number of times the iteration process is carried out, and ``CPU Time'' indicates the computational time measured in seconds. All codes were written by MATLAB R2022a and the computations were performed on a PC equipped with an Intel(R) Core(TM) i7-6700 CPU @ 3.40 GHz 3.41 GHz and 16.00 GB of RAM.

\textbf{Example 6.1.} Let $C=[-1,1]$ and
\begin{equation*}
  F(z)=(1-|z|)z.
\end{equation*}
It follows that $F$ is a quasimonotone and Lipschitz continuous mapping on $C$. And we can get $S=\{-1,0,1\}$ and $S_D=\{0\}$.\\
In Algorithm \ref{AL1}, let $\lambda= 1$, $\mu = 0.3$ and $\xi_n=\frac{100}{(n+1)^{1.1}}$. $u^\ell$ denotes the solution of the problem (\ref{1.1}). We define the stopping criterion as $\text{Error}=\|u_{n + 1}-u_n\|^2$, and set the maximum iteration limit to 500. The proposed algorithm is tested using different initial points $u_1$. The results are shown in the Table \ref{111}.
It is evident that the proposed algorithm is effective under some conditions presented in this paper.
\begin{table}[htp]
\caption{Numerical results of the algorithm for different initial points.}
\begin{tabular}{lllllllll}
\toprule
$u_1$&\multicolumn{2}{c}{\underline{$\text{Error}<10^{-6}$\qquad\qquad}}
&\multicolumn{2}{c}{\underline{$\text{Error}<10^{-8}$\qquad\qquad}}&$u^\ell$\\
&Iter&CPU time&Iter&CPU time&\\
\midrule
0.6 &54&0.0079&73&$0.0082$&0 \\
0.9 &82&0.0070&101&0.0072&0\\
%\midrule
2 &2&0.0023&2&0.0017&-1 \\
%\midrule
3&3&0.0024&3&0.0022&1\\
-3&3&0.0024&3&0.0017&-1\\
random&51&0.0081&70&0.0047&0\\
random&45&0.0034&63&0.0039&0\\
\bottomrule
\end{tabular}
\label{111}
\end{table}

\textbf{Example 6.2.} Let $C=[0,+\infty)$ and
\begin{equation*}
  F(z)=1+\sin(z).
\end{equation*}
Clearly, $F$ satisfies both quasimonotonicity and Lipschitz continuity on $C$. And it follows that $S=\{0\}\cup\{2k\pi+\frac{3\pi}{2},k=0,1,2,\cdots\}$ and $S_D=\{0\}$.\\
The parameters in Algorithm \ref{AL1} are chosen as $\lambda = 1$, $\mu = 0.5$, and $\xi_n = \frac{100}{(n+1)^{1.1}}$. Here, $u^\ell$ refers to the solution of problem (\ref{1.1}).We adopt $\text{Error} = \|u_{n+1} - u_n\|^2$ as the stopping criterion, with a maximum of 500 iterations. The algorithm is evaluated using several different initial values of $u_1$, and the results are summarized in Table \ref{111222}.\\
The results clearly demonstrate the effectiveness of the proposed algorithm under the conditions established in this paper.
\begin{table}[htp]
\caption{Numerical results of the algorithm for different initial points.}
\begin{tabular}{lllllllll}
\toprule
$u_1$&\multicolumn{2}{c}{\underline{$\text{Error}<10^{-6}$\qquad\qquad}}
&\multicolumn{2}{c}{\underline{$\text{Error}<10^{-8}$\qquad\qquad}}&$u^\ell$\\
&Iter&CPU time&Iter&CPU time&\\
\midrule
2 &23&0.0049&29&$0.0053$&0 \\
0.1 &18&0.0051&25&0.0055&0\\
%\midrule
-0.5 &20&0.0046&27&0.0059&0 \\
%\midrule
4&33&0.0061&40&0.0051&0\\
-2&22&0.0039&29&0.0046&0\\
random&21&0.0042&28&0.0049&0\\
\bottomrule
\end{tabular}
\label{111222}
\end{table}

\textbf{Example 6.3.} We consider the following system to apply the algorithm to the signal recovery problem:
\begin{equation}\label{6.1}
  y=Tu+\varepsilon,
\end{equation}
where $y\in\mathbb{R}^M$ is observed signal and $\varepsilon$ is additive noise. The input signal $u$ is mapped to the output signal $y$ by the matrix $T$.
The signal recovery  problem that can be used to solve the problem (\ref{6.1}) is given by: for a constant $\omega>0$,
\begin{equation}\label{6.2}
\min_{u\in\mathbb{R}^N}\frac{1}{2}\|Tu-y\|_2^2\ \text{subject to}\ \|u\|_1\leq\omega.
\end{equation}
The input signal, denoted by $u^\ell$, is generated by randomly placing $K$ nonzero elements (with the value of $\pm1$) in the interval $[-1,1]$.
The sampled data is given by $y=Tu^\ell$, assuming there is no noise present.

For reformulating the problem (\ref{6.2}) as $VI(C,F)$ (\ref{1.1}), we define $C=\{u:\|u\||_1\leq\omega\}$, $\omega=K$ and $F=\nabla g$, where $g(u)=\frac{1}{2}\|Tu-y\|_2^2$. Clearly, $F$ is monotone, and therefore $F$ is also quasimonotone. Let $c(u)=\|u\|_1-\omega$ and $C_n=\{u\in\mathbb{R}^N:c(u_n)\leq \langle \tau_n, u_n-u \rangle\}$. Thus, the projection can be computed by
\begin{equation*}
P_{C_n}(u)=
\begin{cases}
u,&\text{if}\ c(u_n)\leq\langle \tau_n, u_n-u \rangle,\\
u+\frac{\langle\tau_n, u_n-u\rangle-c(u_n) }{\|\tau_n\|^2}\tau_n,&\text{otherwise},
\end{cases}
\end{equation*}
where $\tau_n\in\partial c(u_n)$. And $\partial c(u_n)$ denotes the subdifferential at $u_n$, defined as follows:
\begin{equation*}
{\partial c(u_n)}_i=sign(u_n^{(i)}),
\end{equation*}
where the symbol $sign$ refers to the sign function. And $i$-th component of the vector $u_n$ is denoted by $u_n^{(i)}$.

We set the initial point to be $u_1=(0,\cdots,0)^T\in \mathbb{R}^N$.
For the evaluation of the restoration accuracy, we make use of the following mean squared error:
\begin{equation*}
  \text{Error}=\frac{1}{N}\|u_n-u^\ell\|^2<10^{-6}.
\end{equation*}
We make a comparison between the Algorithm \ref{AL1} presented in this paper and the Algorithm 4.3 from \cite{shehu2019}. For Algorithm \ref{AL1}, we choose $\lambda_1=0.1$, $\mu=0.3$ and $\xi=\frac{100}{(n+1)^{1.1}}$. And for Algorithm 4.3 in \cite{shehu2019}, we choose $\lambda_n=3\times10^{-4}$, $\alpha_n=\frac{1}{10^2n+1}$ and $\gamma=0.5$.

We evaluated the performance of all algorithms under the same conditions by varying $M$, $N$ and $K$ and setting the maximum iteration limit to 2000. The aggregated findings are presented in Table \ref{123}. Additionally, we present the original and recovered signals for cases 2 and 3, shown in Figure \ref{f1}. Figure \ref{222} gives the variation of the ratio with the number of iterations in the condition (A5) presented in this paper, where $\varepsilon=1$ and $Ratio=\frac{|\langle F(z_n), z_n-u^\ell \rangle|}{\|z_n-u^\ell\|^{3}}$. And we can obtain $\inf\limits_{n\in\mathbb{N}}\frac{|\langle F(z_n), z_n-u^\ell \rangle|}{\|z_n-u^\ell\|^{3}}>0$ from Figure \ref{222}.
Finally, we perform a comparison of the Error and the iteration count of all algorithms for the different cases, as illustrated in Figure \ref{fig:12345}. The results demonstrate that our proposed algorithms outperform the others in terms of both accuracy and efficiency.
\begin{table}[htp]
\caption{Numerical results from all algorithms for different $M$, $N$ and $K$.}
\begin{tabular}{lllllllll}
\toprule
Case&M&N&K&\multicolumn{2}{c}{\underline{Algorithm \ref{AL1}\qquad\qquad }}&\multicolumn{2}{c}{\underline{Algorithm 4.3 in \cite{shehu2019}\quad\qquad }}\\
&&&&Iter&CPU time&Iter&CPU time\\
\midrule
Case 1&256&512&20&252&0.5331&835&1.5990 \\
Case 2&256&512&40&371&0.8818&1472&3.0215\\
Case 3&512&1024&60&398&3.0847&757&5.4006 \\
Case 4&512&1024&80&867&6.5861&1856&12.9299 \\
\bottomrule
\end{tabular}
\label{123}
\end{table}
\begin{figure}[htp]
 \begin{minipage}{1\linewidth}
		\subfigure[$M=256,\ N=512,\ K=40$]
{
			%\label{fig:Our 1-AMSIA}
			\includegraphics[width=0.5\linewidth,height=2.1in]{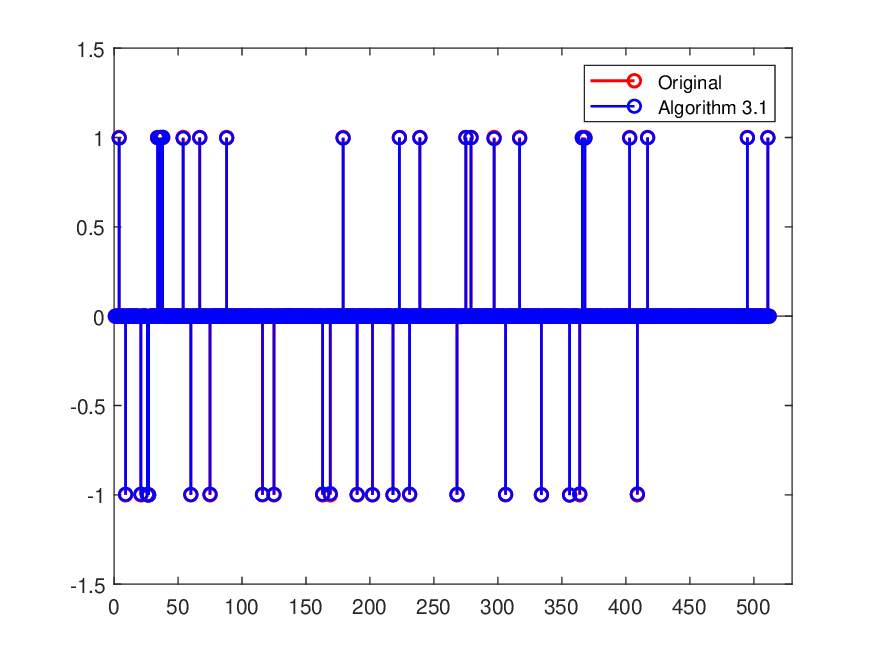}	
		}\noindent
		\subfigure[$M=256,\ N=512,\ K=40$]
{
			%\label{fig:Our 2-AMSIA}
			\includegraphics[width=0.5\linewidth,height=2.1in]{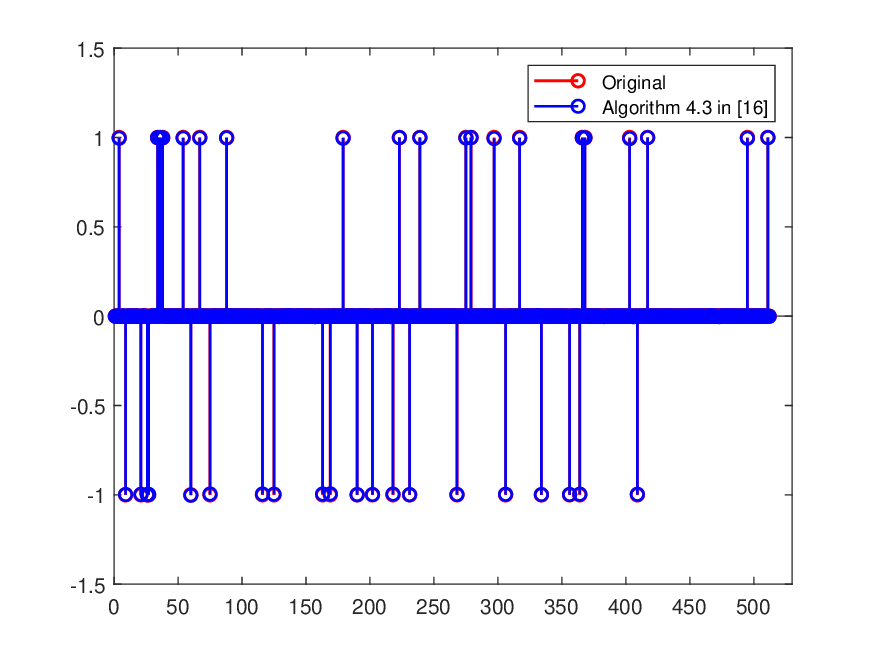}
		}
	\end{minipage}
\begin{minipage}{1\linewidth}	
		\subfigure[$M=512,\ N=1024,\ K=60$]{
			%\label{fig:Our 1-AMSIA}
			\includegraphics[width=0.5\linewidth,height=2.1in]{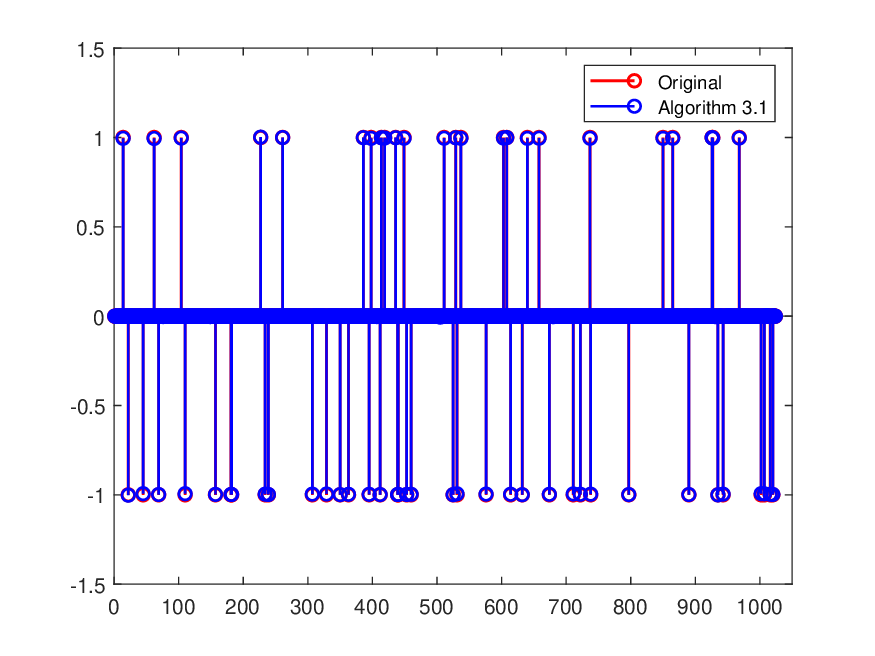}	
		}\noindent
		\subfigure[$M=512,\ N=1024,\ K=60$]{
			%\label{fig:Our 2-AMSIA}
			\includegraphics[width=0.5\linewidth,height=2.1in]{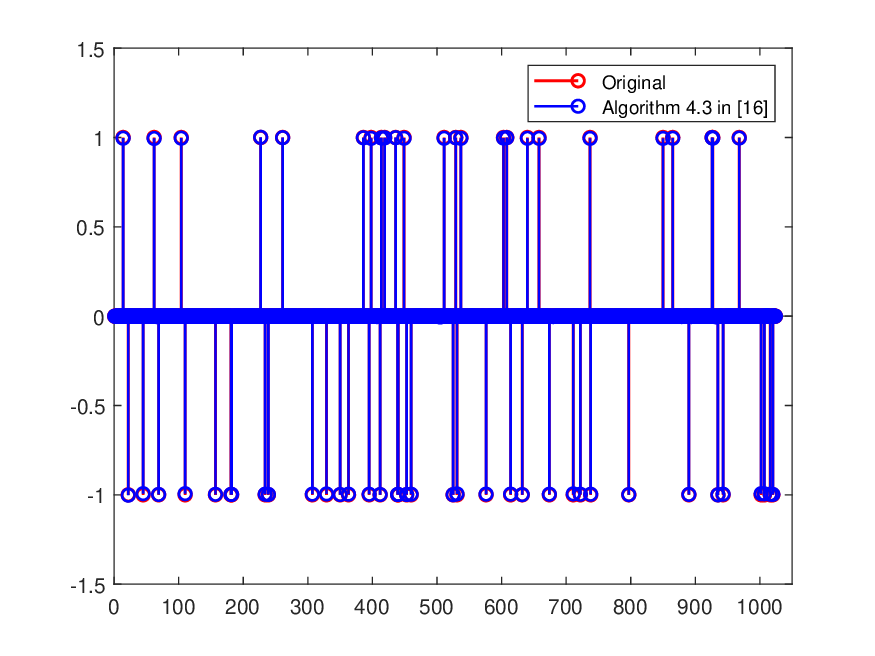}
		}
	\end{minipage}
  \caption{Original signal and recovered signal for different $M$, $N$ and $K$.}
\label{f1}
\end{figure}

\begin{figure}[htp]
	\centering
	\vspace{-0.15in}	
		{
			\includegraphics[width=0.6\linewidth,height=2.6in]{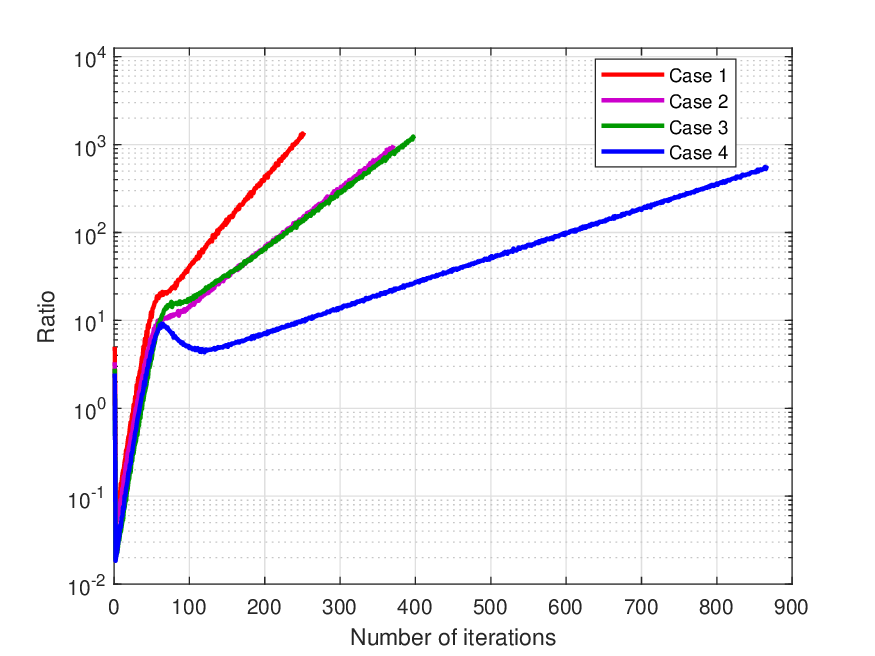}
		}

	\caption{The relationship between the Ratio and the number of iterations.}
	\label{222}
\end{figure}

\begin{figure}[htp]
	\centering
	\vspace{-0.15in}
	\begin{minipage}{1\linewidth}	
		\subfigure[Case 1]{
			\includegraphics[width=0.5\linewidth,height=2.1in]{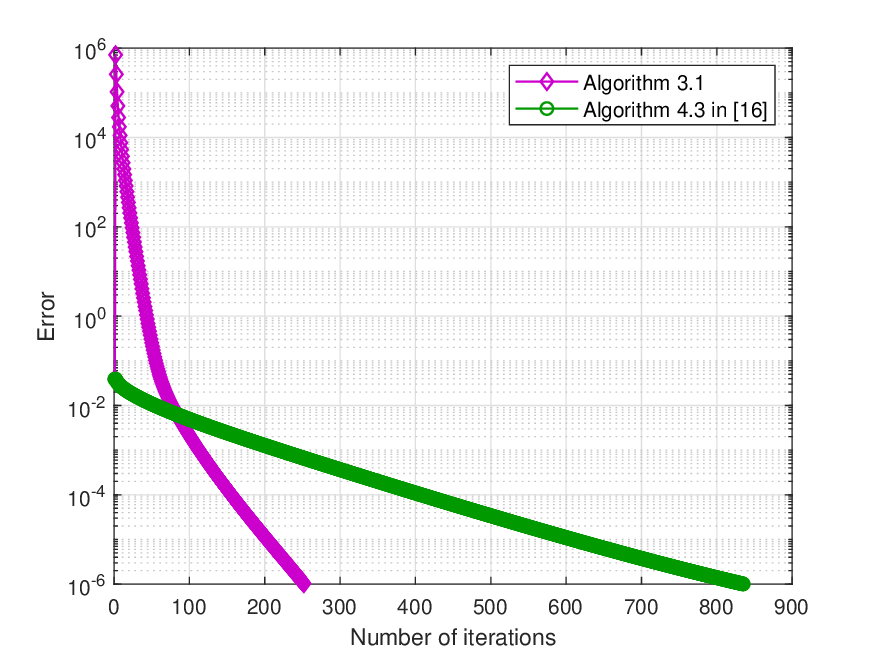}	
		}\noindent
		\subfigure[Case 2]{
			\label{fig:Our 2-AMSIA}
			\includegraphics[width=0.5\linewidth,height=2.1in]{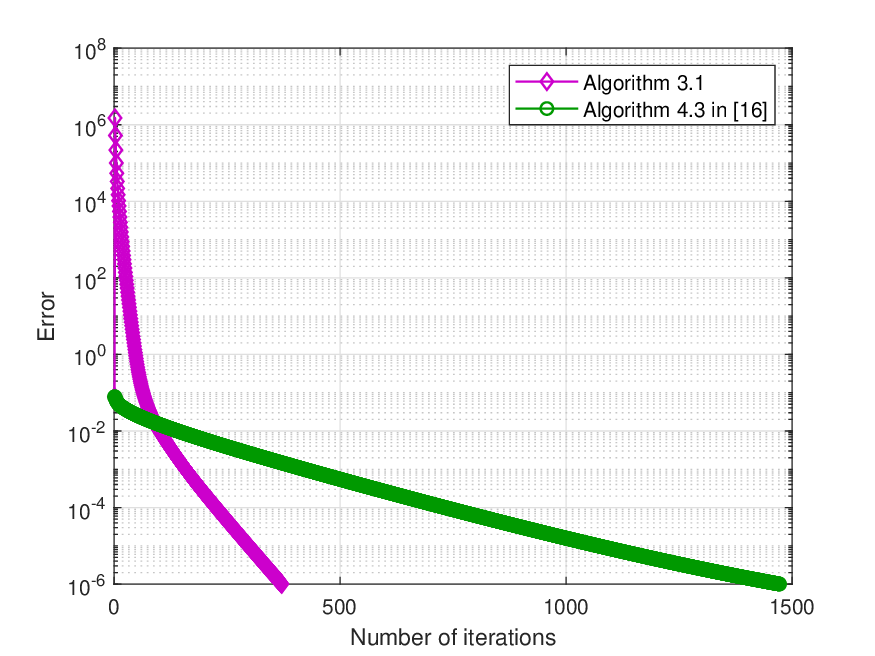}
		}
	\end{minipage}
	\vskip -0.3cm
	\begin{minipage}{1\linewidth }
		\subfigure[Case 3]{
			\label{fig:Our 3-AMSIA}
			\includegraphics[width=0.5\linewidth,height=2.1in]{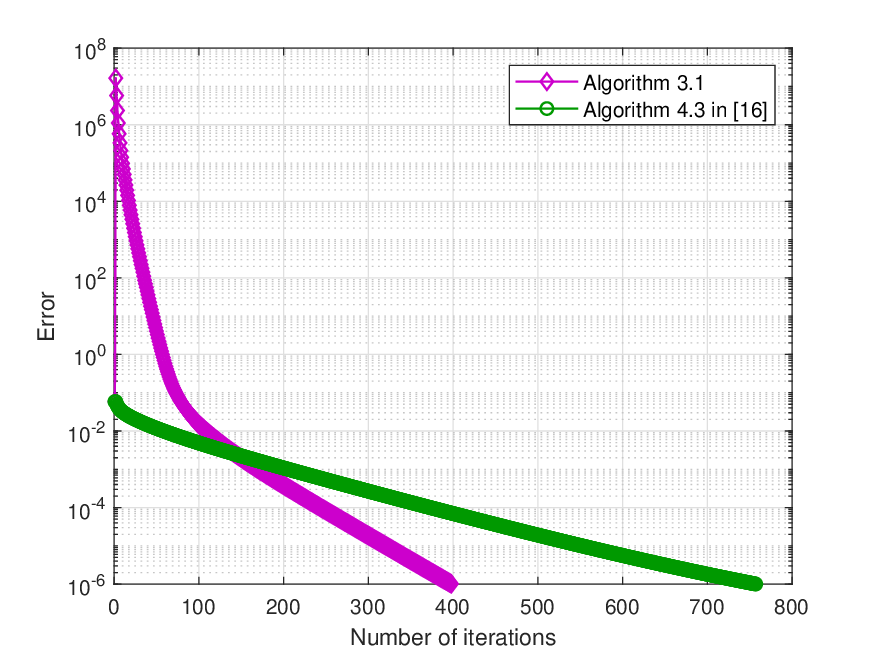}
			
		}\noindent
		\subfigure[Case 4]{
			\label{fig:JSKS Alg}
			\includegraphics[width=0.5\linewidth,height=2.1in]{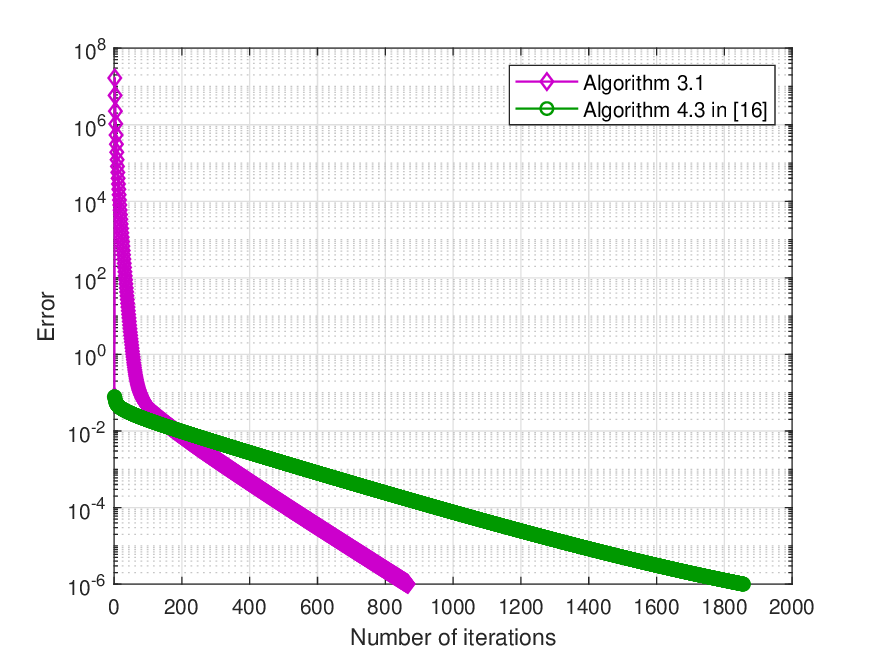}
			}
		
	\end{minipage}

	\caption{Comparison of Error and number of iterations.}

	\label{fig:12345}
\end{figure}
%\section{Numerical Experiment}\label{sec4}

\section{Conclusion}\label{sec6}
This paper investigates the variational inequality problem with a quasimonotone operator based on the self-adaptive Tseng's extragradient algorithm proposed in \cite{yang3}. We improve the condition in \cite{yang3}, where $A=\{p\in C: F(p)=0\}\backslash S_D$ is assumed to be a finite set, to make it applicable to cases where the set contains infinitely many points. Subsequently, we establish the result that the sequence generated by the algorithm converges strongly to a point in the solution set \( S \) of the variational inequality problem. Furthermore, we analyze the sublinear convergence and $Q$-linear convergence rate under some mild conditions. In order to verify the validity and implementability of the proposed algorithm and condition, we conducted numerical experiments. The results in this paper generalize and improve the relevant literature on quasimonotone variational inequality problems.

\bmhead{Acknowledgements}
The authors express their appreciation to the responsible editor and anonymous reviewers for their careful reading.

\bmhead{Funding} This work was supported by the National Natural Science Foundation of China (No. 12261019), Natural Science Basic Research Program Project of Shaanxi Province (No. 2024JC-YBMS-019), the Natural Science Foundation of Shaanxi Province of China (No. 2023-JC-YB-049), the Fundamental Research Funds for the Central Universities, and the Innovation Fund of Xidian University (No. YJSJ25009).
\section*{Data Availability} Data sharing is not applicable to this article as no datasets were generated or analyzed during the current study.
\section*{Declarations}

\textbf{Conflict of interest} The authors declare no conflicts of interest.

\bigskip
\bibliography{sn-bibliography}

\end{document}